\documentclass[sn-mathphys-num]{sn-jnl}
\usepackage{graphicx}
\usepackage{amsmath,amssymb,amsthm,commath,mathtools,mathrsfs,url,hyperref,tabularx}
\usepackage[monochrome]{xcolor}
\usepackage{float}
\usepackage{lmodern}

\theoremstyle{thmstyleone}%
\newtheorem{theorem}{Theorem}%
\newtheorem{proposition}[theorem]{Proposition}%
\newtheorem{lemma}[theorem]{Lemma}
\newtheorem{corollary}[theorem]{Corollary}
\theoremstyle{thmstyletwo}%
\newtheorem{example}{Example}%
\newtheorem{remark}{Remark}%
\theoremstyle{thmstylethree}%
\newtheorem{definition}{Definition}%
\newcommand{\trace}{\mathrm{tr}}

\newcommand{\volume}{\mathrm{vol}}
\newcommand{\adjugate}{\mathrm{adj}}
\newcommand{\transpose}{\mathsf{T}}
\newcommand{\conv}{\mathrm{conv}}
\newcommand{\LCI}{\mathrm{LCI}}
\newcommand{\CGclos}{\mathrm{CG}\text{-}\mathrm{cl}}

\newcommand{\bigzero}{\mbox{\normalfont\Large\bfseries 0}}
\newcommand{\rvline}{\hspace{-\arraycolsep}\vline\hspace{-\arraycolsep}}
\newcommand\ZZ{\mathbb{Z}}
\newcommand\RR{\mathbb{R}}

\renewcommand\>{\rangle}
\newcommand\<{\langle}

\newcommand{\ba}{\mathbf{a}}
\newcommand{\bb}{\mathbf{b}}
\newcommand{\bc}{\mathbf{c}}
\newcommand{\be}{\mathbf{e}}

\newcommand{\bu}{\mathbf{u}}
\newcommand{\bv}{\mathbf{v}}

\newcommand{\bw}{\mathbf{w}}
\newcommand{\bx}{\mathbf{x}}
\newcommand{\by}{\mathbf{y}}
\newcommand{\bz}{\mathbf{z}}
\newcommand{\ext}{\operatorname{ext}}

\makeatletter
\def\defcal#1{%
\expandafter\newcommand\csname cal#1\endcsname{\mathcal{#1}}}
\count@=0
\loop
\advance\count@ 1
\edef\y{\@Alph\count@}%
\expandafter\defcal\y
\ifnum\count@<26
\repeat
\def\defbb#1{%
\expandafter\newcommand\csname bb#1\endcsname{\mathbb{#1}}}
\count@=0
\loop
\advance\count@ 1
\edef\y{\@Alph\count@}%
\expandafter\defbb\y
\ifnum\count@<26
\repeat
\def\defbb#1{%
\expandafter\newcommand\csname bf#1\endcsname{\mathbf{#1}}}
\count@=0
\loop
\advance\count@ 1
\edef\y{\@alph\count@}%
\expandafter\defbb\y
\ifnum\count@<26
\repeat
\makeatother

\begin{document}
\title[Integer Points in Arbitrary Convex Cones]{Integer Points in Arbitrary Convex Cones: \\The Case of the PSD and SOC Cones
\footnote[3]{This article provides more details on the extended abstract published at IPCO 2024 \cite{10.1007/978-3-031-59835-7_8}, including a resolution of a conjecture posed in the extended abstract.}} 

%

\author[1]{\fnm{Jes\'{u}s A.} \sur{De Loera}}\email{jadeloera@ucdavis.edu}
\author[1]{\fnm{Brittney} \sur{Marsters}}\email{bmmarsters@ucdavis.edu}
\author*[1]{\fnm{Luze} \sur{Xu}}\email{lzxu@ucdavis.edu}
\author[2]{\fnm{Shixuan} \sur{Zhang}}\email{shixuan.zhang@tamu.edu}

\affil*[1]{University of California, Davis, Davis CA 95616, USA}
\affil[2]{Texas A\&M University, College Station TX 77843, USA}
\abstract{
We investigate the semigroup of integer points inside a convex cone. We extend classical results in integer linear programming to integer conic programming. We show that the semigroup associated with nonpolyhedral cones can sometimes have a notion of finite generating set with the help of a group action. We show this is true for the cone of positive semidefinite matrices (PSD) and the second-order cone (SOC). Both cones have a finite generating set of integer points, similar in spirit to Hilbert bases, under the action of a finitely generated group. We also extend notions of total dual integrality, Gomory-Chv\'{a}tal closure, and Carath\'{e}odory rank to integer points in arbitrary cones.
\footnotetext{This research is partially based upon work supported by the National Science Foundation under Grant No. DMS-1929284 while the first, third, and fourth authors were in residence at the Institute for Computational and Experimental Research in Mathematics in Providence, RI, during the Discrete Optimization program. The first and second authors were partially supported by the National Science Foundation under grants No.\ DMS 1818969, No.\ DMS 2434665,
and No.\ DMS 2348578. We thank Kurt Anstreicher, Greg Blekherman, Santanu Dey, Chiara Meroni, Bento Natura, Pablo Parrilo, Renata Sotirov, and Andy Sun for comments and support.}
 }

\maketitle              
%
%
%
%

\section{Introduction}

A semigroup $S$ is a subset of $\mathbb{Z}^N$ that contains $\mathbf{0}$ and is closed under addition. Given a convex cone $C\subseteq \mathbb{R}^N$, the integer points $S_C:=C\cap\mathbb{Z}^N$ form a semigroup which we will call the \emph{conical semigroup} of $C$. In particular, given any compact convex body $K\subseteq\RR^n$, the integer points $\text{cone}(K\times \{1\})\cap\ZZ^{n+1}$ form a conical semigroup. Conical semigroups appear not just in optimization~\cite{MR3835599,berndt2023new}, but also in algebra and number theory~\cite{MR4391780,MR3633776}. Given a convex cone \(C\subseteq\bbR^N\) for \(N\ge1\), we say a subset \(B\subseteq S_C\) is an \emph{integral generating set} of $S_C$ if for any \(s\in S_C\) there exist \(b_1,\dots,b_m\in B\) and \(c_1,\dots,c_m\in\bbZ_{\ge0}\) such that \(s=\sum_{i=1}^{m}c_ib_i\), for some \(m\ge1\). Furthermore, we call $B$ a \emph{conical Hilbert basis} if $B$ is an inclusion-minimal integral generating set. 

When the defining cone $C$ is rational, polyhedral, and pointed, there is abundant literature on the topic. It is well-known that we have a unique finite Hilbert basis in this case \cite{de2013algebraic,schrijver1998theory}. Historically, Hilbert bases have been fundamental in the theory and algorithms of combinatorial optimization. For example, determining if a rational system $A\bx\leq \bb$ is totally dual integral (TDI) is equivalent to checking if, for every face $F$ of the polyhedron $P:=\{\bx:A\bx\leq \bb\}$, the rows of $A$ which are active in a face $F$ form a Hilbert basis for cone($F$)\cite{schrijver1998theory}. 

It is natural to ask, what properties transfer from rational polyhedral cones to arbitrary convex cones? The question we will consider is \emph{do we preserve any notion of finiteness in generating sets for semigroups when we relax the polyhedral condition and instead consider general conical semigroups? Are there Hilbert bases for general cones?}

Unlike for polyhedra, prior work on understanding the lattice points of general convex cones is very sparse. For instance, when is the convex hull of the lattice points a closed set? This has been investigated in \cite{dey_properties_2013}.
In \cite{hemmecke_representation_2007} the authors presented conditions for when a set of lattice points have a \emph{finite} integral generating set, partially emulating polyhedra. Unfortunately, this is rarely the case. Our paper discusses finite generation for conical semigroups and partially extends the polyhedral cone theory of Hilbert bases to nonpolyhedral convex cones. 

Our first contribution to the subject is to introduce a group action, allowing for finite generation with the help of the 
group. Our main results will show that with a group action, 
the semigroups of the cone of positive semidefinite matrices and 
(some of) the second-order cone are in fact finitely generated. Both cones play a key role in modern optimization \cite{barvinok2002course,bental2001lectures}. We also discuss 
some applications of our nonpolyhedral point of view.

In what follows, we denote $\mathrm{GL}(N,\ZZ):=\{U\in\bbZ^{N\times N}:~|\det(U)|=1\}$. 
Here is our new notion of finite generation for conical semigroups.



\begin{definition}
Given a conical semigroup $S_C\subset\ZZ^N$, we call it $(R,G)$-finitely generated if there is a finite subset $R\subseteq S_C$ and a finitely generated subgroup $G\subseteq \mathrm{GL}(N, \ZZ)$ acting on $C$ linearly such that 
\begin{enumerate}
    \item both the cone $C$ and the semigroup $S_C$ are invariant under the group action, i.e., $G\cdot C=C$ and $G\cdot S_C =S_C$, and
    \item every element $s\in S_C$ can be represented as \[s =  \displaystyle\sum_{i\in K}\lambda_i g_{i}\cdot r_i\] for some $r_i \in R$, $g_{i}\in G$, and $\lambda_i \in \mathbb{Z}_{\geq 0}$, $i\in K$, where $K$ is a finite index set.
\end{enumerate}
\end{definition}

When $C$ is a (pointed) rational polyhedral cone, then the conical semigroup $S_C=C\cap\ZZ^N$ is $(R, G)$-finitely generated if we take $R$ to be its Hilbert basis, and $G$ to be the trivial group $\{I_N\}$ of the identity action. Similarly, note that if $S_C$ is an $(R,G)$-finitely generated semigroup, then $\cup_{r\in R}(G\cdot r)$ is an integral generating set of $S_C$, which is a superset of a conical Hilbert basis.  We call $R$ the set of \emph{roots} of $S_C$, and $\cup_{r\in R}(G\cdot r)$ the set of generators for $S_C$. 

While a nonpolyhedral cone cannot be finitely generated in the usual sense, 
using an infinite (yet finitely generated) group $G$ allows us to extend our understanding beyond the polyhedral case. Because the possibly infinite generators for $S_C$ can be obtained by group action $G$ on a finite set $R$ and $G$ is finitely generated, this allows for the possibility of algorithmic methods.
The well-known Krein-Milman theorem states that any point in a closed pointed cone $C$ can be generated by extreme rays, denoted by $\ext(C)$ \cite{barvinok2002course}. When we restrict to the conical semigroup $S_C$ and nonnegative integer combinations, the primitive integer points on the extreme rays of $C$ must be contained in the set of generators of $S_C$, where an integer point $x=(x_1,\dots,x_N)\in\ZZ^N$ is \emph{primitive} if $\mathrm{gcd}(x_1,\dots,x_N)=1$. We call the integer points of $S_C$ on the extreme rays of $C$ \emph{extreme points}, denoted by $\ext(S_C):=\{y\in S_C:y\in \ext(C)\cap\ZZ^N\}$.
However, as in the polyhedral case, the generators will often include extra nonextreme boundary points or even interior points. We provide the following definition of \emph{sporadic points} which are points that cannot have an extreme point subtracted from them and still remain within the cone.


\begin{definition}\label{def:sporadic}
We say a point $x\in S_C=C\cap \ZZ^N$ is sporadic if there does not exist $y\in \ext(S_C)$ such that $x - y \in S_C$.
\end{definition}
If $x\in S_C$ is sporadic, then $x$ cannot be written as an integer conical combination of extreme points (even though it can be written as a real combination of them). From the definition of sporadic points, we know that all points $x\in S_C$ can be written as an integer conical combination of primitive extreme points and one sporadic point. To show that a semigroup is $(R,G)$-finitely generated, it is sufficient to show that the set of primitive extreme points and sporadic points are finite or can be obtained from a finitely generated group $G$ that acts on a finite set of roots, $R$.

In this work, we are mainly interested in the positive semidefinite cone (PSD) and the second-order cone (SOC).
In Sections~\ref{sec:psd} and~\ref{sec:soc} of this paper, we will present the following two main results pertaining to integer points in the PSD cone, denoted $\calS^n_+(\mathbb{Z})$, and the integer points in the second order cone, denoted $\mathrm{SOC}(n)\cap\ZZ^n$.

\begin{theorem}\label{thm:PSD}
    The conical semigroup of the cone of $n\times n$ positive semidefinite matrices, $\calS^n_+(\mathbb{Z})$, is $(R,G)$-finitely generated by $G\cong\mathrm{GL}(n,\ZZ)$ where $G$ acts on $X\in\calS^n_+(\ZZ)$ by $X\mapsto UXU^\transpose$ for each $U\in\mathrm{GL}(n,\ZZ)$, and by $R$, the union of any single rank-one matrix and a finite subset of the sporadic points. Moreover,
    \begin{enumerate}
        \item If $n\le 5$, then there are no sporadic points. Thus, $R = \{\be_1\be_1^\transpose\}$, where $\be_1$ is the first unit vector.
        \item If $n= 6$, then $R = \{\be_1\be_1^\transpose, M\}$, where $M$ is a single sporadic point defined in Section~\ref{sec:psd} Proposition~\ref{example:n=6}.
    \end{enumerate}
\end{theorem}
We say that two matrices $X_1, X_2$ are unimodularly equivalent if $X_2 = U\cdot X_1=UX_1U^\transpose$ for some $U\in \mathrm{GL}(n,\ZZ)$. It is easy to see that it defines an equivalence relation for all integer PSD matrices.
Note that the equivalence class of $\be_1\be_1^\transpose$ are all rank-1 integer matrix $\bx\bx^\transpose$ for some primitive integer vector $\bx\in\ZZ^n$.
An interpretation of Theorem~\ref{thm:PSD} is that for dimension $n\le 5$, every integer PSD matrix can be represented as the sum of rank-$1$ matrices $\bx\bx^\transpose$ for some primitive integer vector $\bx\in\ZZ^n$.
However, the same result fails for dimension $n = 6$. In this case, we will have that every integer PSD matrix can be represented as the sum of rank-$1$ matrices and one sporadic matrix $Y$, which is unimodularly equivalent to $M$ (this matrix was first found by~\cite{mordell_representation_1937}).
In general, every integer PSD matrix can be represented as the sum of rank-$1$ matrices and one sporadic matrix, which is unimodularly equivalent to a matrix in the finite set $R$.
Regarding prior work that inspired us, we mention that \cite{letchford_binary_2012} contains a similar rank-1 decomposition structure for PSD $\{0,1\}$ matrices: a PSD $\{0,1\}$ matrix $X\in \calS^n_+(\mathbb{Z})\cap\{0,1\}^{n\times n}$ satisfies $X=\sum_{i\in K} \bx_i \bx_i^\transpose$ for $\bx_i\in\{0,1\}^n$, where $K$ is a finite index set.
Similarly, \cite{demeijer2023integrality} extends the results to PSD $\{0,\pm1\}$ matrices:  a PSD $\{0,\pm1\}$ matrix $X\in \calS^n_+(\mathbb{Z})\cap\{0,\pm1\}^{n\times n}$ satisfies $X=\sum_{i\in K} \bx_i \bx_i^\transpose$ for $\bx_i\in\{0,\pm1\}^n$, where $K$ is a finite index set. Our results extend to all integer positive semidefinite matrices. 

For the second-order cone family we are able to prove a similar theorem for dimensions up to $n=10$ by extending the construction of the Barning-Hall tree in \cite{PTmatrixgen} for the primitive extreme points (or Pythagorean tuples) to classify the sporadic points. 

\begin{theorem}\label{thm:SOC}
    For dimension $3\le n\leq 10$, the conical semigroup $\mathrm{SOC}(n)\cap\ZZ^n$ is $(R,G)$-finitely generated. The group $G$ and the roots $R$ will be defined in Section~\ref{sec:soc}.
\end{theorem}



While it might be tempting to believe that all conical semigroups are $(R,G)$-finitely generated for some finite set $R$ and some subgroup $G\subseteq \mathrm{GL}(N,\mathbb{Z})$, this is not true in general.
In Section \ref{notallconesrgfg},  Example~\ref{ex:NonFinGen} shows a non-rational cone in the plane such that its semigroup is not $(R,G)$-finitely generated. 
For future work we leave open the following natural questions: \emph{which convex cones are (R,G)-finitely generated? Can we characterize them, at least for the plane?}




What is the significance of these results beyond their connections to classical geometry of numbers, lattices, and number theory? (See e.g., \cite{Kaveh+Khovanskii:semigroups}.) We motivate our interest about conical semigroups with two applications in optimization. In what follows, we assume that our cone $C\subset\bbR^N$ is full-dimensional. 


The first application regards the notion of \emph{Chv\'atal-Gomory cuts} which is useful in {\color{blue}branch-and-cut} methods for integer programming.  How much of this can be extended to conic integer programming? Given a linear map $\mathcal{A}:\bbR^m\to\bbR^N$ and $\bc\in\bbR^N$, we define a \emph{linear conical inequality} (LCI) system as 
\[
    \LCI_C(\bc,\calA):=\{\bx\in\bbR^m:\bc-\mathcal{A}(\bx)\in C\}
\]
where $\bc\in\bbZ^N$ and $\mathcal{A}(\bbZ^m)\subseteq\bbZ^N$.
When $C$ is the cone of positive semidefinite matrices in $\calS^n(\bbR)$, then $N=\binom{n+1}{2}$ and
\(
    \calA(\bx)=\sum_{i=1}^{m}x_i A_i
\)
for some matrices $A_1,\dots,A_m\in\calS^n(\bbZ)$.
This is known as a \emph{linear matrix inequality} and defines a \emph{spectrahedron}.
An important concept for LCI is called total dual integrality (TDI), which has been well-known for polyhedral cones $C$~\cite{giles1979total,edmonds1984total} and recently extended to spectrahedral cones~\cite{decarlisilver2020notion,demeijer2022chvatal}.
We use $C^*$ to denote the dual cone of $C$, $\calA^*$ to denote the adjoint linear map of $\calA$, and give a definition for general cones here.
\begin{definition}
    An LCI system \(\bc-\calA(\bx)\in C\) is totally dual integral, if for any \(\bb\in\bbZ^m\), the dual optimization problem
    \begin{align*}
        \min\quad & y(\bc)\quad
        \mathrm{s.t.}\quad \calA^*(y)=\bb, \ 
                           y\in C^*,
    \end{align*}
    whenever feasible, has an integer optimal solution \(y^*\in C^*\cap\bbZ^N\).
\end{definition}

To approximate the convex hull of \(Z:=\LCI_C(\bc,\calA)\cap\bbZ^m\), a commonly used approach (quite similar to its polyhedral version) is to add \emph{Chvátal-Gomory} (CG) cuts, which are defined as follows~\cite{demeijer2022chvatal}. If \(\bu\in\bbZ^m\) is an integral vector and \(v\in\bbR\) a real number such that the linear inequality
$    \bu^\transpose \bx\le v$
is \emph{valid} for all \(x\in\LCI_C(c,\calA)\), then the inequality 
$ \bu^\transpose \bx\le \lfloor v\rfloor$
is valid for all \(\bx\in Z\) and called a CG cut.
There are possibly infinitely many CG cuts so we define the (elementary) CG closure as
\begin{equation}
\CGclos(Z):=\bigcap_{\substack{(\bu,v)\in\bbZ^m\times\bbR:\\S\subseteq\{\bx:\bu^\transpose \bx\le v\}}}\left\{\bx\in\bbR^m:\bu^\transpose \bx\le\lfloor v\rfloor\right\}.
\end{equation}

Now take any linear function $w\in C^*$ such that $w(\bbZ^N)\subseteq\bbZ$.
Then, a CG cut can be generated by
\begin{equation*}
    w\circ\calA(\bx)\le \lfloor w(\bc)\rfloor,
\end{equation*}
as, by definition, \(w\circ\calA(\bbZ^m)\in\bbZ\).
Conversely, if the conical semigroup $S_{C^*}:=C^*\cap\bbZ^N$ is $(R,G)$-finitely generated, then we can get all CG cuts through $R$ and $G$ for our TDI LCI system. This is one of the nice consequences of this property.
\begin{theorem}\label{thm:TDI}
    Suppose $C\subset\bbR^N$ is a full-dimensional convex cone such that $S_{C^*}:=C^*\cap\bbZ^N$ is $(R,G)$-finitely generated, and $\LCI_C(\bc,\calA)$ is TDI.
    Then the CG closure for $Z:=\LCI_C(\bc,\calA)\cap\bbZ^m$ can be described by
    \[
        \CGclos(Z)=\left\{\bx\in\bbR^m:(g\cdot r)^\transpose\calA(\bx)\le\lfloor(g\cdot r)^\transpose \bc\rfloor,\quad\forall\,r\in R,g\in G\right\}.
    \]
\end{theorem}

The final application has to do with classical notions of integer rank \cite{cook1986integer}. 
Just like the notion of (real) rank of a linear system allows us to bound the number of nonzero entries in a solution of a linear system, we want to know how many elements are needed to decompose any element of a conical semigroup as a linear combination of generators with nonnegative integer coefficients. Suppose that our conical semigroup $S_C=C\cap\ZZ^N$ has an integer generating set $B$. For any element \(s\in S_C\), there exist integer generators \(b_1,\dots,b_m\in B\) and \(\lambda_1,\dots,\lambda_m\in\bbZ_{\ge1}\) such that \(s=\sum_{i=1}^{m}\lambda_ib_i\), for some \(m\ge1\).
The minimum number \(m\) needed in the sum is called the \emph{integer Carathéodory rank} (ICR) of \(s\), and the maximum number over all \(s\in S_C\) is the ICR of the conical semigroup \(S_C\) or the cone \(C\). We show an upper bound on the ICR that depends only on the dimension \(N\).
The proof is almost identical to the popular polyhedral result in~\cite{cook1986integer,sebo1990hilbert} but we must use the extreme point characterization of semi-infinite linear optimization~\cite{charnes1963duality} to allow infinite generating sets.

\begin{theorem}\label{thm:ICR}
    Let \(C\subset\bbR^N\) be an arbitrary pointed convex cone and \(S_C:=C\cap\bbZ^N\).
    Then \(\mathrm{ICR}(S_C)\le 2N-2\).
\end{theorem}


\section{The Positive Semidefinite (PSD) Cone}\label{sec:psd}
Let \(\calS^n(\bbZ)\) (resp.\ \(\calS^n(\bbR)\)) denote the set of \(n\times n\) symmetric matrices of integer (resp.\ real) entries.
For a matrix \(X\in\calS^n(\bbZ)\), we say that \(X\) is PSD (denoted as \(X\succeq0\)) if and only if it is so when regarded as a real matrix \(X\in\calS^n(\bbR)\). We denote $\calS^n_+(\bbZ)$ as the set of integer PSD matrices.

The group \(\mathrm{GL}(n,\ZZ)\) embeds into \(\mathrm{GL}(N,\ZZ)\) as follows.
Given a matrix \(U\in\mathrm{GL}(n,\ZZ)\) and any \(X\in\calS^n(\bbZ)\), we define the action \(U\cdot X:=UXU^\transpose\).
This action is a linear map and takes integer points in \(\ZZ^N\) to integer points, and thus can be represented by the multiplication with a matrix in \(\mathrm{GL}(N,\ZZ)\).
{\color{blue}
It is well-known that this group \(\mathrm{GL}(n,\ZZ)\) is finitely generated by three matrices~\cite{trott1962pair}:
$$
\begin{bmatrix}
    0 & 1 & 0 & \dots & 0 & 0\\
    0 & 0 & 1 & \dots & 0 & 0\\
     &  &  & \dots &  & \\
    0 & 0 & 0 & \dots & 0 & 1\\
    1 & 0 & 0 & \dots & 0 & 0\\
\end{bmatrix},
\begin{bmatrix}
    1 & 0 & 0 & \dots & 0 & 0\\
    1 & 1 & 0 & \dots & 0 & 0\\
     &  &  & \dots &  & \\
    0 & 0 & 0 & \dots & 1 & 0\\
    0 & 0 & 0 & \dots & 0 & 1\\
\end{bmatrix},
\begin{bmatrix}
    0 & 1 & 0 & \dots & 0 & 0\\
    1 & 0 & 0 & \dots & 0 & 0\\
     &  &  & \dots &  & \\
    0 & 0 & 0 & \dots & 1 & 0\\
    0 & 0 & 0 & \dots & 0 & 1\\
\end{bmatrix}.
$$
}
For simplicity, we keep the notation \(U\cdot X\) for the action of \(U\in\mathrm{GL}(n,\ZZ)\) on \(X\) without explicitly specifying its representation in \(\mathrm{GL}(N,\ZZ)\).

\subsection{Lemmas for $n\le 5$ and $n=6$}
The following \emph{integer rank-1 decomposition} for PSD integer matrices is studied in \cite{mordell_representation_1937}.
We recast their arguments with a modern geometric perspective, and use it to extend the notion of $(R,G)$-finite generation to the PSD cone.
\begin{lemma}\label{lem:n<=5}
    If $n\le 5$, then for any $X\in \calS_+^n(\ZZ)$, we can find a finite index set $K$ and vectors $\bx_i\in\ZZ^n$, $i\in K$ such that
    \begin{equation}\label{eqn:rank1}
        X = \sum_{i\in K} \bx_i\bx_i^\transpose.
    \end{equation}
\end{lemma}
%
To restate Definition~\ref{def:sporadic} in the PSD case, we say an integer matrix \(X\in\calS^n(\bbZ)\) is \emph{sporadic} if there does not exist \(\bx\in\bbZ^n\setminus\{\mathbf{0}\}\) such that \(X-\bx\bx^\transpose\succeq0\).
Lemma~\ref{lem:n<=5} is equivalent to the fact that there is no sporadic point in $\calS_+^n(\ZZ)$ when $n\le 5$.

\begin{proposition}\label{prop:finite_termination}
    There is no sporadic point in $\calS_+^n(\ZZ)$ if and only if every positive semidefinite integer matrix in $\calS_+^n(\ZZ)$ has an integer rank-1 decomposition.
\end{proposition}
\begin{proof}
    If there is no sporadic point in $\calS_+^n(\ZZ)$, then for every $Y\in\calS_+^n(\ZZ)$, there exists \(\bx\in\bbZ^n\setminus\{\mathbf{0}\}\) such that \(Y-\bx\bx^\transpose\succeq0\).
    For $X\in\calS_+^n(\ZZ)$, we do the following procedure for \(X_0:=X\) (with index \(i\) initialized to 1).
    \begin{enumerate}
        \item Take any \(\bx_i\in\bbZ^n\setminus\{\mathbf{0}\}\) such that \(X_i:=X_{i-1}-\bx_i\bx_i^\transpose\succeq0\).
        \item If \(X=0\), then we have found an integer rank-1 decomposition \(X=\sum_{j=1}^{i}\bx_j\bx_j^\transpose\);
        otherwise set the index \(i\leftarrow i+1\) and go back to step 1.
    \end{enumerate}
    To see that the procedure terminates in finitely many steps, note that the diagonal of \(\bx_i\bx_i^\transpose\) contains at least 1 nonzero entry because \(\bx_i\neq0\).
    Thus the trace \(\trace(X_i)\le\trace(X_{i-1})-1\) for any \(i\ge1\) because the entries are integers.
    The procedure can repeat no more than \(\trace(X)\) times as \(\trace(X_i)\ge0\).
    
    If every $X\in\calS_+^n(\ZZ)$ has an integer rank-1 decomposition $X=\sum_{i\in K}\bx_i\bx_i^\transpose$, then any of \(\bx_i, i\in K\) satisfies the requirement \(X-\bx_i\bx_i^\transpose\succeq0\).
\end{proof}

For any matrix \(X\in\calS^n(\bbR)\), we can define a convex set \(K(X):=\{\bx\in\bbR^n:X-\bx\bx^\transpose\succeq0\}\).
Since \(X-\bx\bx^\transpose\succeq0\) if and only if, for any \(\bv\in\bbR^n\), \(\vert \bv^\transpose \bx\vert^2\le \bv^\transpose X\bv\), we see that \(K(X)\) is a compact convex set that is symmetric about the origin.
This provides another equivalent formulation of the integer rank-1 decomposition.

\begin{proposition}\label{prop:geometric_sporadic}
    For $X\in \calS_+^n(\ZZ)$, $X$ is sporadic if and only if $K(X)\cap \ZZ^n=\{\mathbf{0}\}$.
\end{proposition}
This provides a geometric perspective to our problem.
Note that the set \(K(X)\) is an ellipsoid {\color{blue}(that may fail to be full-dimensional)} because
$$
X\succeq \bx\bx^\transpose
\iff\begin{bmatrix}1 & \bx^\transpose\\ \bx & X\end{bmatrix}\succeq0
\iff \bx^\transpose X^\dagger \bx\le 1, (I-XX^\dagger)\bx=0,
$$
by the positive semidefiniteness of Schur complements, where \(X^\dagger\) denotes the pseudoinverse of \(X\).
In the case where \(X\) has full rank, 
\begin{equation*}
    K(X)=\{\bx\in\bbR^n:\bx^\transpose\adjugate(X)\bx\le\det(X)\}\text{ with }\volume(K(X))=V_n\sqrt{\det(X)}
\end{equation*}
where $\adjugate(X)$ is the adjugate of $X$ satisfying $\adjugate(X) = \det(X)X^{-1}$ and \(V_n:=\pi^{n/2}/\Gamma(\frac{n}{2}+1)\) is the volume of the unit \(n\)-ball.
Note that $K(X)$ is not necessarily full-dimensional if $X$ is rank deficient.
The degenerate case for rank 1 is characterized by the following proposition.
\begin{proposition}\label{prop:rank1}
    Suppose $X\in \calS_+^n(\bbZ)$, and $\mathrm{rank}(X)=1$, then $X = \lambda \bx\bx^\transpose$, where $\bx\in\mathbb{Z}^n$ and $\lambda\in\mathbb{Z}_{\ge 1}$.
\end{proposition}
\begin{proof}
    As $X\succeq0$ and $\mathrm{rank}(X)=1$, we can assume that $X=\ba\ba^\transpose$ for some $\ba\in\mathbb{R}^n$. Because $X\in\calS^n(\bbZ)$, we have $a_ia_j\in\mathbb{Z}$ for $i,j\in[n]$. In particular, $a_i^2\in\mathbb{Z}$. Denote $k_i:=a_i^2\in\mathbb{Z}_{\ge 0}$. 
    Without loss of generality, we can assume that $k_i\ge 1$, i.e., $a_i\ne 0$, otherwise, we can just consider the submatrix corresponding to the nonzero $k_i$.

    Suppose that there exists some $a_i\in\mathbb{Z}\setminus \{0\}$ and $a_j\in\{\pm\sqrt{k_j}\}\notin\mathbb{Q}$. Then $a_ia_j\notin\mathbb{Q}$, a contradiction. Therefore, we must have either $a_i\in \mathbb{Z}$ for all $i$ or $a_i\notin \mathbb{Q}$ for any $i$. 
    
    If $a_i\in\mathbb{Z}$ for all $i$, then the result holds with $\lambda=1$ and $x=a$.
    
    If $a_i\notin \mathbb{Q}$ for all $i$, i.e., $k_i$ is not a square. 
    Because $a_ia_j\in\mathbb{Z}$, we have $\sqrt{k_ik_j}\in\mathbb{Z}$, which implies that $k_ik_j=t_{ij}^2$ for some integer $t_{ij}$. Suppose that $p_1,\dots, p_s$ are all the prime factors in the decompositions of $k_i$, $i\in[n]$ such that $k_i=\prod_{\ell=1}^s p_\ell^{\alpha_\ell^i}$, for some $\alpha_\ell^i\in\mathbb{Z}_{\ge 0}$. We have $k_ik_j=\prod_{\ell=1}^s p_\ell^{\alpha_\ell^i+\alpha_\ell^j}=t_{ij}^2$, which implies that $\alpha_\ell^i+\alpha_\ell^j$ is even. Therefore, for a fixed $\ell$, either $\alpha_\ell^i$ is even for all $i\in[n]$ or $\alpha_\ell^i$ is odd for all $i\in[n]$. Let $I:=\{\ell\in[s]: \alpha_\ell^i~\text{is odd}\}$ and $\lambda =\prod_{\ell\in I} p_\ell$. We have $k_i/\lambda$ is a square, thus $X=\lambda \bx\bx^\transpose$ for $x=a/\sqrt{\lambda}$, where $ x_i=\sqrt{k_i/\lambda}\in\mathbb{Z}$.
\end{proof}
From Proposition~\ref{prop:rank1}, we can directly prove the case for $n=2$ using Minkowski's First Theorem (for example, see \cite{de2013algebraic}): if $\mathrm{vol}(K(X))>4$, then there is a nonzero integral point in $K(X)$.
\begin{proposition}\label{prop:n=2}
    Lemma~\ref{lem:n<=5} holds for $n=2$.
\end{proposition}
\begin{proof}
For $n=2$. Let $X = \begin{bmatrix} a_{11} & a_{12}\\ a_{12} & a_{22}\end{bmatrix}\succeq 0$, where $a_{11},a_{12},a_{22}\in\mathbb{Z}$, which implies that $a_{11}\ge0, a_{22}\ge0, a_{11}a_{22}-a_{12}^2\ge0$. By Proposition~\ref{prop:rank1}, we only need to consider the case when $X\succ 0$, i.e., $\det(X)=a_{11}a_{22}-a_{12}^2\ge1, a_{11}\ge 1, a_{22}\ge 1$.
Then, 
$$
\begin{aligned}
K(X)&:=\{\bx\in\mathbb{R}^2: X-\bx\bx^\transpose\succeq 0\}\\
&=\{\bx\in\mathbb{R}^2: a_{11}-x_1^2\ge 0, a_{22}-x_{2}^2\ge0, (a_{11}-x_1^2)(a_{22}-x_2^2)-(a_{12}-x_1x_2)^2\ge0\}.
\end{aligned}
$$
We claim that $K(X)=\{\bx\in\mathbb{R}^2: (a_{11}-x_1^2)(a_{22}-x_2^2)-(a_{12}-x_1x_2)^2\ge0\}.$ 
To see this, note that 
\begin{align*}
    0\le(a_{11}-x_1^2)(a_{22}-x_2^2)-(a_{12}-x_1x_2)^2 &= (a_{11}a_{22} - a_{12}^2) - a_{22}x_1^2 +2 a_{12}x_1x_2 - a_{11}x_2^2\\
    & = \frac{a_{11}a_{22} - a_{12}^2}{a_{11}}(a_{11} - x_1^2) - a_{11}(x_2-\frac{a_{12}}{a_{11}}x_1)^2
\end{align*}
which together with $a_{11},a_{22},a_{11}a_{22}-a_{12}^2>0$ implies that $a_{11}-x_1^2\ge0$.
Similarly,
\begin{align*}
    0\le(a_{11}-x_1^2)(a_{22}-x_2^2)-(a_{12}-x_1x_2)^2 = \frac{a_{11}a_{22} - a_{12}^2}{a_{22}}(a_{22} - x_2^2) - a_{22}(x_1-\frac{a_{12}}{a_{22}}x_2)^2
\end{align*}
implies that $a_{22}-x_2^2\ge0$, which shows the claim.

Since $X$ has full rank, $K(X)=\{\bx\in\mathbb{R}^2: \det(X) - a_{22}x_1^2 +2 a_{12}x_1x_2 - a_{11}x_2^2\ge0\}$ is a centrally symmetric ellipsoid with area $\pi \sqrt{\det(X)}$.

If $\det(X)\ge 2$, then $\mathrm{vol}(K(X))\ge \sqrt{2}\pi> 4$, we know that $K(X)\cap\mathbb{Z}^2\ne\emptyset$ by Minkowski's First Theorem.

If $\det(X)=1$, then $K(X)=\{x\in\mathbb{R}^n: a_{22}x_1^2 -2 a_{12}x_1x_2 + a_{11}x_2^2\le1\}$. Because $a_{22},a_{12},a_{11}\in\mathbb{Z}$, we have $K(X)\cap\mathbb{Z}^2 = \tilde{K}(X)\cap\mathbb{Z}^2$, where $\tilde{K}(X)=\sqrt{2-\epsilon}\cdot K(X)=\{x\in\mathbb{R}^n: a_{22}x_1^2 -2 a_{12}x_1x_2 + a_{11}x_2^2\le 2-\epsilon\}$ for any $0<\epsilon<1$. Therefore, $\mathrm{vol}(\tilde{K}(X))=(2-\epsilon)\cdot\pi>4$ when we take any $\epsilon<2-\frac{4}{\pi}$. Therefore, $\tilde{K}(X)\cap\mathbb{Z}^2$ is nonempty by Minkowski's First Theorem so neither is $K(X)\cap\ZZ^2$.
\end{proof}
%
In the degenerate case, we can reduce the problem to one involving full-rank matrices of some lower dimension.
\begin{lemma}\label{lem:reduce_rank}
Let $X \in\calS^n(\bbZ)$.
If $r=\mathrm{rank}(X)<n$, then $X$ is unimodularly equivalent to
$$
\begin{bmatrix}
    \mathbf{0} & \mathbf{0}\\
    \mathbf{0} & \hat{X}
\end{bmatrix},
$$
for some $\hat{X}\in\mathbb{Z}^{r\times r}$, $\mathrm{rank}(\hat{X})=r$.
\end{lemma}
\begin{proof}
    If $\mathrm{rank}(X)<n$, then there exists a primitive vector $\bz\in\mathbb{Z}^n$ in the subspace $N:=\{\by\in\bbR^n:~X\by=\bf0\}$.
    The sublattice $\Lambda:=\ZZ^n\cap N$ is primitive and thus a basis of $\Lambda$ containing $\bz$ can be extended to a basis of $\bbZ^n$, which we denote as $U=[\bz,\bu_2,\dots,\bu_n]$~\cite[Chapter 2, Lemma 4]{nguyen2010lll}. Because $U$ is a basis of $\bbZ^n$, we know that $|\det(U)|=1$, i.e., $U$ is unimodular. Then
    $$
    U^\transpose X U = \begin{bmatrix}
    0 & 0\\
    0 & \hat{X}
    \end{bmatrix}.
    $$
    Iterating this process until $\hat{X}$ is positive definite, i.e., $\mathrm{rank}(\hat{X})=r$.
\end{proof}
Notice that if $X_1, X_2$ are unimodularly equivalent, then $K(X_1)\cap\bbZ^n\ne\{\mathbf{0}\}$ if and only if $K(X_2)\cap\bbZ^n\ne\{\mathbf{0}\}$.
Thus our problem expects an answer under the unimodular equivalence of integer matrices in $\calS_+^n(\ZZ)$.

The scaling of $K(X)$ into $\tilde{K}(X)$ (while preserving the integer points) in the proof for Proposition~\ref{prop:n=2} results in
\[\volume(\tilde{K}(X))<V_n\sqrt{\det(X)}\cdot\biggl(\frac{\det(X)+1}{\det(X)}\biggr)^{n/2}\] 
where $V_n$ is a constant dependent only on the dimension $n$, and the right-hand side can be approached arbitrarily.
When \(n=3\), \(V_n\approx 4.189\), the right-hand side becomes \(2^{3/2}\approx 2.828\), \(\sqrt{2}\cdot(3/2)^{3/2}\approx 2.598\), \(\sqrt{3}\cdot(4/3)^{3/2}\approx 2.667\) for \(\det(X)=1,2,\) and \(3\), respectively, and greater than 2 for \(\det(X)\ge 4\).
Thus \(\volume(\tilde{K}(X))>8\) so \(\tilde{K}(X)\cap\bbZ^3\neq\{0\}\) by Minkowski's First Theorem.

To prove Lemma~\ref{lem:n<=5}, we need to use a more sophisticated method based on the \emph{Hermite constant}~\cite{schurmann2009computational}
$$
\gamma_n:=\left(\max_{A\succ 0}\frac{\lambda_1(A)}{(\det(A))^{\frac{1}{n}}}\right)^n,\text{ where }\lambda_1(A) = \min_{\bx\in\bbZ^n\setminus\{\bf0\}}(\bx^\transpose A \bx).
$$
%
%
%
\begin{remark}
Hermite gives a bound $\gamma_n\le(\frac43)^{\frac{n(n-1)}{2}}$. The exact value of $\gamma_n$ is only known for $n\le 8$ and $n=24$. 

\vspace{1mm}
\begin{center}
    \centering
    \begin{tabularx}{0.5\textwidth}{X|XXXXXXXX}
    \hline
    $n$ & $2$ & $3$ & $4$ & $5$ & $6$ & $7$ & $8$ & $24$ \\
    \hline
    $\gamma_n$ & $\frac{4}{3}$ & $2$ & $2$ & $8$ & $\frac{64}{3}$ & $64$ & $256$ & $4^{24}$\\
    \hline
    \end{tabularx}
\end{center}
\vspace{1.5mm}
\end{remark}
\begin{remark}
From a volume argument by Minkowski's First Theorem, we have 
$$
\min_{\bx\in\bbZ^n\setminus\{\bf0\}}(\bx^\transpose A \bx)\le \frac{4}{\pi}\Gamma(1+\frac{n}{2})^{\frac{2}{n}}\det(A)^{\frac{1}{n}}\sim \frac{2n}{\pi e}\det(A)^{\frac1n},
$$
which is better than the bound given by $\gamma_n\le(\frac43)^{\frac{n(n-1)}{2}}$ when $n$ is large. But this estimation on $\gamma_n$ is not enough to prove Lemma \ref{lem:n<=5} for the dimension $n=4,5$.
\end{remark}
%
\begin{proof}[Proof of Lemma~\ref{lem:n<=5}]
The case $n=1$ follows from Proposition~\ref{prop:rank1}. 
We will show that $K(X)\cap\bbZ^n\ne\{\mathbf{0}\}$ for $2\le n\le 5$, where \(K(X)=\{\bx\in\bbR^n:\bx^\transpose\adjugate(X)\bx\le\det(X)\}\).
By the definition of the Hermite constant, we have
$$
\min_{\bx\in\bbZ^n\setminus\{\bf0\}}(\bx^\transpose \mathrm{adj}(X) \bx)\le (\gamma_n \det(\mathrm{adj}(X)))^{\frac{1}{n}}=(\gamma_n (\det(X))^{n-1})^{\frac{1}{n}}.
$$
For $n=2,3,4,5$, we have
$\frac{n^n}{(n-1)^{n-1}}>\gamma_n$ as
$\frac{2^2}{1^1}=4$, $\frac{3^3}{2^2}\approx 6.75$, $\frac{4^4}{3^3}\approx 9.48$, $\frac{5^5}{4^4}\approx 12.21$, and $\frac{6^6}{5^5}\approx 14.93$.
By taking the derivative with respect to $\det(X)$ for finding extremum value, we know that $\frac{(\det(X)+1)^n}{(\det(X))^{n-1}}\ge \frac{n^n}{(n-1)^{n-1}}$. Thus,
$\gamma_n<\frac{(\det(X)+1)^n}{\det(X)^{n-1}}$.
Therefore,
$$
\min_{\bx\in\bbZ^n\setminus\{\bf0\}}(\bx^\transpose \mathrm{adj}(X) \bx)\le (\gamma_n \det(\mathrm{adj}(X)))^{\frac{1}{n}}=(\gamma_n \det(X)^{n-1})^{\frac{1}{n}}<\det(X) +1.
$$
Because $x^\transpose\adjugate(X)x,
\det(X)\in\mathbb{Z}$ for any $x\in\mathbb{Z}^n$, we have
$$
\min_{\bx\in\bbZ^n\setminus\{\bf0\}}(\bx^\transpose \mathrm{adj}(X) \bx)\le\det(X),
$$
which implies that $K(X)\cap(\mathbb{Z}^n\setminus\{\mathbf{0}\})\ne\emptyset$.
Lemma~\ref{lem:n<=5} now follows from Propositions~\ref{prop:geometric_sporadic} and~\ref{prop:finite_termination}, and Lemma~\ref{lem:reduce_rank}.
\end{proof}
%
The argument used to prove Lemma~\ref{lem:n<=5} fails for $n\ge 6$, but it implies that the determinant of the sporadic matrices is bounded by a constant only dependant on $n$.
For example, in the case of $n=6$, the argument only fails when $3\le \det(X)\le 14$;
for $n=7$, it only fails when $2\le \det(X)\le 56$,
and for $n=8$, it only fails when $1\le \det(X)\le 247$.
We summarize this observation in the following corollary.
\begin{corollary}\label{cor:bounded_det}
    If \(X\in\calS_+^n(\bbZ)\) is sporadic, then \(\det(X)<\gamma_n\).
\end{corollary}
A sporadic matrix for $n=6$ was initially found in~\cite{mordell_representation_1937}. 
%
\begin{proposition}\label{example:n=6}
In $n=6$, the matrix $M$ is sporadic, i.e., $K(M)\cap\mathbb{Z}^n = \{\mathbf{0}\}$.
$$
M=
\left[
\begin{array}{c@{\hskip 1em}c@{\hskip 1em}c@{\hskip 1em}c@{\hskip 1em}c@{\hskip 1em}c}
    2 & 0 & 1 & 1 & 1 & 1 \\[-0.5em]
    0 & 2 & 0 & 1 & 1 & 1 \\[-0.5em]
    1 & 0 & 2 & 1 & 1 & 1 \\[-0.5em]
    1 & 1 & 1 & 2 & 1 & 1 \\[-0.5em]
    1 & 1 & 1 & 1 & 2 & 1 \\[-0.5em]
    1 & 1 & 1 & 1 & 1 & 2
\end{array}
\right]
    \text{ with } \det(M)=3.
$$
\end{proposition}
\begin{proof}
We verify that $\min_{\bx\in\bbZ^n\setminus\{\bf0\}}(\bx^\transpose \mathrm{adj}(X) \bx)>\det(X) = 3$.
$$\mathrm{adj}(X)=(\det(X)) X^{-1}=
\begin{bmatrix}
    4 & 3 & 1 & -2 & -2 & -2\\
    3 & 6 & 3 & -3 & -3 & -3\\
    1 & 3 & 4 & -2 & -2 & -2\\
    -2 & -3 & -2 & 4 & 1 & 1\\
    -2 & -3 & -2 & 1 & 4 & 1\\
    -2 & -3 & -2 & 1 & 1 & 4
\end{bmatrix}
$$
Note that $\bx^\transpose \mathrm{adj}(X) \bx = [(x_1+2x_2+x_3-x_4-x_5-x_6)^2 + (x_1+x_2+x_3-x_4-x_5-x_6)^2+x_2^2]+[(x_1-x_3)^2+x_1^2+x_3^2]+[(x_4-x_5)^2+(x_4-x_6)^2+(x_5-x_6)^2]$.
Let $A_2:=(x_1+2x_2+x_3-x_4-x_5-x_6)^2 + (x_1+x_2+x_3-x_4-x_5-x_6)^2+x_2^2$, $A_{13}:=(x_1-x_3)^2+x_1^2+x_3^2$ and $A_{456}:=(x_4-x_5)^2+(x_4-x_6)^2+(x_5-x_6)^2$. Then $\bx^\transpose \mathrm{adj}(X) \bx= A_2+A_{13}+A_{456}$.

Suppose, there exists $x\in\mathbb{Z}^6$ such that $\bx^\transpose \mathrm{adj}(X) \bx\le 3$. We are going to show that $x=0$. Notice that $A_2,A_{13},A_{456}$ are even, which implies that $A_2+A_{13}+A_{456}\le 2$. Then at most one of $A_2,A_{13},A_{456}$ is nonzero.

We consider the following three cases:
\begin{enumerate}
    \item if $A_{13}=0, A_{456}=0$, then $x_1=x_3=0, x_4=x_5=x_6=0$. Because $A_2=6x_2^2\le 2$, we have $x_2=0$.
    \item if $A_{2}=0, A_{456}=0$, then $x_4=x_5=x_6=0$, $x_2=0$, $x_1+x_3=0$. Because $A_{13}=6x_1^2\le 2$, we have $x_1=0$.
    \item if $A_{2}=0, A_{13}=0$, then $x_1=x_3=0$, $x_2=0$, $x_4+x_5+x_6=0$. Because $A_{456}=(x_4-x_5)^2+(2x_4+x_5)^2+(x_4+2x_5)^2=6(x_4^2+x_5^2+x_4x_5)\le 2$, we have $x_4=x_5=x_6=0$.
\end{enumerate}
Therefore, $\min_{\bx\in\bbZ^n\setminus\{\bf0\}}(\bx^\transpose \mathrm{adj}(X) \bx)>3$. For $\bx=\be_1$, $\bx^\transpose \mathrm{adj}(X) \bx=4$, i.e., $\min_{\bx\in\bbZ^n\setminus\{\bf0\}}(\bx^\transpose \mathrm{adj}(X) \bx)=4$.
\end{proof}

Moreover, in~\cite{ko1939decomposition}, it is shown that for $n=6$, $M$ is the unique sporadic matrix under unimodular equivalence. Using this fact, we have the following Lemma.

\begin{lemma}\label{lem:n=6}
    If $n=6$, then for any $X\in \calS_+^n(\ZZ)$,
    \begin{equation*}
        X = \sum_{i\in K} \bx_i\bx_i^\transpose + Y
    \end{equation*}
    for $\bx_i\in\ZZ^n$ and $Y$ unimodularly equivalent to $M$, where $K$ is a finite index set.
\end{lemma}

\subsection{Proof of Theorem~\ref{thm:PSD}}
\begin{proof}
    We know that the primitive extreme points are generated from the group $\mathrm{GL}(n,\ZZ)$ that acts on $\{\be_1\be_1^\transpose\}$. 
    Thus we only need to prove that the sporadic points are generated from the group $\mathrm{GL}(n,\ZZ)$ on a finite set $R$.
    
    Corollary~\ref{cor:bounded_det} shows that for any sporadic matrix \(X\), $\det(X)<\gamma_n$.
    By~\cite[Theorem 2.4]{schurmann2009computational}, there exists a constant \(\alpha_n>0\) depending only on \(n\), such that for any positive definite matrix \(X\in\calS^n(\bbZ)\), there is a unimodularly equivalent matrix \(X'\) of \(X\) with diagonal entries satisfy
    \[
        \prod_{i=1}^{n}X'_{ii}\le\alpha_n\det(X')=\alpha_n\det(X)<\alpha_n\gamma_n.
    \] 
    Because \(X'\in\calS^n(\bbZ)\) is positive definite, \(X'_{ii}\ge 1\) and thus is bounded from above.
    From this we see that there are only finitely many possibilities for such \(X'\) because each off-diagonal entry must satisfy \(\vert X'_{ij}\vert^2\le X'_{ii}X'_{jj}\) for any \(1\le i,j\le n\).

    The two special cases $n\le 5$ and $n=6$ follow from Lemma~\ref{lem:n<=5} and~\ref{lem:n=6}.
\end{proof}


\section{The Second-Order Cone (SOC)}\label{sec:soc}


In this section, we use $T_n$ to denote the conical semigroup $\text{SOC}(n)\cap\ZZ^n$ where 
\[\text{SOC}(n):= \left\{ \bx\in\RR^n : 0\leq \sqrt{x_1^2+\dots+x_{n-1}^2}\leq x_n\right\}.\]
The last component $x_n$ of $\bfx\in T_n$ is referred to as the \emph{height} of $\bfx$.
Additionally, for $\ba,\bb\in\RR^n$, consider the bilinear form 
\begin{align*}\label{quadratic form}
    \<\ba,\bb\> := a_1b_1 +a_2b_2 +\dots +a_{n-1}b_{n-1} - a_nb_n.
\end{align*}
In this quadratic space, the reflection in vector $\bw$ is defined as $\bx\rightarrow \bx - 2\frac{\<\bx, \bw\>}{\<\bw,\bw\>}\bw$.

\begin{definition}\label{A_n definition}
     Let $P_{i,j}$ be the permutation matrix that swaps the $i$-th and $j$-th columns and define $Q_k$ be the matrix determined by 
    {\small\begin{equation*}
        (Q_k)_{i,j} =
        \begin{cases}
            -1 & \text{ if } i=j=k \\
            1  & \text{ if } i=j\neq k\\
            0 & \text{ if } i\neq j.
        \end{cases}
    \end{equation*}}

    For $n=3$, let $A_3$ denote the matrix associated with the reflection in the vector $(1,1,1)$.
    For $4 \leq n \leq 10$, let $A_n$ denote the matrix associated with the reflection in the vector $(1,1,1,0,\dots,0,1)$ also associated to this bilinear form: 
    {\small\[
    A_3 = 
    \begin{pmatrix}
    -1 & -2 & 2 \\ -2 & -1 & 2 \\ -2 & -2 & 3 \\    
    \end{pmatrix}, \hspace{3mm}
    A_n 
    = \begin{pmatrix}
            \begin{matrix}
                0 & -1 & -1 \\
                -1 & 0 & -1 \\
                -1 & -1 & 0
            \end{matrix}
        & \rvline & \bigzero  & \rvline 
        & \begin{matrix} 1 \\ 1\\ 1 \end{matrix} \\   
        \hline \bigzero & \rvline 
        & \text{I}_{n-4} &\rvline & \bigzero \\\hline
        \begin{matrix} -1 & -1 & -1 \end{matrix} 
        & \rvline &\bigzero &\rvline & 2 
    \end{pmatrix} 
    \]  }
    We define the matrix $A_n^+ = Q_1Q_2\dots Q_{n-1} A_n$.  Note that $A_n$ is unimodular.
\end{definition}


Elements $\bfs\in T_n$ such that $\<\bfs,\bfs\>=\bf0$ belong to the boundary of $T_n$, and we will denote the set of these points as $\partial T_n$. In number theory, these points are called Pythagorean tuples. 
In \cite{PTmatrixgen}, they proved that the set of primitive Pythagorean tuples, denoted as $\ext^p(T_n)$, is generated by finitely many matrices acting on a finite set $R$ for $3\leq n\leq 10$.

\begin{lemma}[Theorem 1 in \cite{PTmatrixgen}]\label{cass (R,G)}
    For $3\leq n \leq 10$, $\ext^p(T_n) = \cup_{r\in R} G\cdot r$, where the group 
    \[G=\left\< A_n, Q_1, \dots, Q_{n-1}, P_{1,2}, P_{1,3},\dots P_{1,n-1} \right\>\] and the sets 
    \begin{enumerate}
        \item $R =\left\{(1,0,\dots,0,1)^\transpose\right\}$ for $3\leq n <10$,
        \item $R=\left\{(1,0,0,0,0,0,0,0,0,1)^\transpose, (1,1,1,1,1,1,1,1,1,3)^\transpose\right\}$ for $n=10$,
    \end{enumerate}
    where $G$ acts on $R$ by left multiplication. 
\end{lemma}
\noindent
We remark that any action in $G$ in the above lemma maps $\partial T_n$ to $\partial T_n$. 

We will begin this section by discussing some structures of the sporadic points of $T_n$. Then we use these structures of Pythagorean tuples and sporadic points to prove Theorem \ref{thm:SOC}.

\subsection{Sporadic Points of $\text{SOC}(n)\cap \ZZ^n$}

In this section, we will begin by restating the definition of sporadic points in the case of $\text{SOC}(n)$ and offer two partial characterizations of sporadic elements of $\text{SOC}(n)$.

\begin{definition}
    We call a point $\bfs\in T_n$ \textit{sporadic} if there is no integral point $\bfp$ such that $\<\bfp, \bfp\>=0$ and $\bfs-\bfp\in T_n$.
\end{definition}

Just as the group $G$ takes elements of $\partial T_n$ to $\partial T_n$, the group $G$ will take sporadic elements to sporadic elements. This closure ensures that our action by $G$ on the semigroup $T_n$ is well-defined. 


\begin{lemma}
    Let $\bfs\in T_n$. 
    \begin{enumerate}
        \item Then, $A_n^+ \bfs$ and $(A_n^+)^{-1}\bfs$ are both in $T_n$. \label{cone_closure}
        \item If $\bfs$ is sporadic, then $A_n^+ \bfs$ and $(A_n^+)^{-1}\bfs$ are both sporadic.\label{sporadic_closure}
    \end{enumerate}
\end{lemma}

\begin{proof}
    The first claim follows by simply checking the required inequalities directly.  For the second claim, we proceed by contradiction. Suppose $\bfs$ is sporadic but $(A_n^+) \bfs$ is not. Then, there is some point $\bfp\in T_n$ such that $\<\bfp,\bfp\>=0$ and $(A_n^+)\bfs - \bfp\in T_n$. However, we would then have that \[(A_n^+)^{-1}(A_n^+ \bfs - \bfp) = \bfs - (A_n^+)^{-1} \bfp \in T_n.\] As $\bfp\in\partial T_n$, $(A_n^+)^{-1} \bfp\in \partial T_n$, which is a contradiction with $\bfs$ being sporadic. The case of $(A_n^+)^{-1}\bfs$ being sporadic follows the same argument.
\end{proof}

Next, we will provide some lemmas about sporadic points that are necessary to prove Theorem~\ref{thm:SOC}. 
Lemmas~\ref{ceiling} and~\ref{sporadic=-1} show that sporadic points are close to the boundary $\partial T_n$.

\begin{lemma}\label{ceiling}
    Suppose $\bfs \in T_n$ is a primitive sporadic with nonnegative entries such that $s_n>1$ and $s_i\neq 0$ for some $i\in\{1,\dots,n-1\}$. Then, {\small\[s_n = \left\lceil\sqrt{s_1^2+s_2^2+\dots+s_{n-1}^2}\right\rceil\]}
\end{lemma}

\begin{proof}
    We will show this by proving that \[\sqrt{s_1^2+s_2^2+\dots+s_{n-1}^2}< s_n < \sqrt{s_1^2+s_2^2+\dots+s_{n-1}^2} +1\] where the first inequality is given by membership in $T_n$. Without loss of generality, we can assume that $s_1\neq 0$. By way of contradiction, suppose that $\bfs$ is a primitive sporadic such that $s_n>1$ and $s_1> 0$, and that $s_n \ge \sqrt{s_1^2+s_2^2+\dots+s_{n-1}^2} +1$. Then, we would have that  \[s_n - 1 \ge \sqrt{s_1^2+s_2^2+\dots+s_{n-1}^2} > \sqrt{(s_1-1)^2+s_2^2+\dots+s_{n-1}^2}\] which is equivalent to $\bfs- (1,0,\dots,0,1)\in T_n$. This contradicts the assumption that $s$ is sporadic. Thus, we have have the desired equality.
\end{proof}

\begin{lemma} \label{sporadic=-1}
    Let $\bfs\in T_n$. If $\<\bfs,\bfs\>=-1$, then $\bfs$ is sporadic.
\end{lemma}    


\begin{proof}
By way of contradiction, suppose $\<\bfs,\bfs\>=-1$ and that $\bfs$ is not sporadic. Then, there exists some $\textbf{p}\in T_n$ such that $\<\bfp,\bfp\>=0$ and $\textbf{s}-\textbf{p}\in T_n$. This is equivalent to saying that 
\[\<\bfs-\bfp, \bfs-\bfp\>\leq 0.\] This gives us that 
\begin{align*}
    \<\bfs,\bfs\> -2\<\bfs,\bfp\> +\<\bfp,\bfp\> =-1-2\<\bfs,\bfp\> &\leq 0.
\end{align*}
Thus, $\<\bfs,\bfp\>\geq -\frac12$ so $\<\bfs,\bfp\>\geq0$  by its integrality. However, as $\<\bfs,\bfs\>=-1$ implies that $\sqrt{s_1^2+\dots+s_{n-1}^2}<s_n$ and $\<\bfp,\bfp\>=0$ implies that $\sqrt{p_1^2+\dots+p_{n-1}^2}=p_n$, using the Cauchy-Schwarz inequality, we have that \[s_1p_1+\dots+s_{n-1}p_{n-1}\le\sqrt{(s_1^2+\dots+s_{n-1}^2)(p_1^2+\dots+p_{n-1}^2)} <s_np_n.\]  Thus, $\<\bfs,\bfp\><0$, reaching a contradiction. Therefore, $\<\bfs,\bfs\>=-1$ implies that $\bfs$ is sporadic.
\end{proof}    


Inspired by the structure of Pythagorean tuples, we analyze the set of sporadic points that remain at the same height in $T_n$ after multiplication by $(A_n^+)^{-1}$. Let $(p)_n$ denotes the $n^{\text{th}}$ coordinate of $p$, 

\begin{lemma} \label{lemma: stable height}
     Let $n\leq 10$. Suppose $\bfs\in T_n$ is a primitive sporadic such that $s_1\geq\dots\geq s_{n-1}\geq 0$ and $s_{n}>1$. The following list of tuples are the only such $\bfs$ where $((A_n^+ )^{-1}\bfs)_n= s_n$.

     \begin{itemize}
         \item For $n=7$, we have the following tuple: $(1,1,1,1,1,1,3)$.
         \item For $n=8$, we have the following tuples: $(1,1,1,1,1,1,1,3)$, $(1,1,1,1,1,1,0,3)$.
         \item For $n=9$, we have the following tuples:
         {\small
            \begin{align*}
                (1,1,1,1,1,1,1,1,3), (1,1,1,1,1,1,1,0,3),
                (1,1,1,1,1,1,0,0,3), (2,2,2,2,2,2,2,1,6).
            \end{align*} 
        }
         \item For $n=10$, we have the following tuples:
            {\small\begin{align*}
                (1,1,1,1,1,1,1,1,0,3), &\hspace{3mm} (1,1,1,1,1,1,1,0,0,3), \hspace{3mm} 
                (1,1,1,1,1,1,0,0,0,3), \\
                (2,2,2,2,2,2,2,2,1,6), &\hspace{3mm}(2,2,2,2,2,2,2,1,0,6).
            \end{align*} }
     \end{itemize}
\end{lemma}

\begin{proof}
    This is equivalent to showing that these are the only such sporadic points such that $s_1+s_2+s_3=s_{n}$. As $\bfs$ is sporadic, $\bfs-(1,0,\dots,0,1)\notin T_n$. This is equivalent to saying that $(s_{n}-1)^2< (s_1-1)^2+s_s^2 +\dots+ s_{n-1}^2$ or 
    \begin{align} \label{fixed height ineq}
        2s_1s_2+2s_2s_3+2s_1s_3-2s_2-2s_3-s_4^2 -\dots - s_{n-1}^2< 0.
    \end{align}
    We begin by showing that the first six coordinates must be equal. We proceed by contradiction in each of the below arguments.
    \begin{itemize}
        \item Suppose that $s_1\geq s_2+1$. Then, \eqref{fixed height ineq} implies 
        \begin{align*}
            0 &> 2s_2(s_2+1)+2s_2s_3+2(s_2+1)- 2s_2 -2s_3 -s_4^2-\dots-s_{n-1}^2 \\
            & = 2s_2^2 + 4s_2s_3 - s_4^2-\dots-s_{n-1}^2 \geq 0.
        \end{align*} 
        As this is a contradiction, we must have that $s_1=s_2$.

        \item Suppose that $s_2\geq s_3+1$. Then, \eqref{fixed height ineq} implies  
        \begin{align*}
            0 &>  2s_1(s_3+1) +2s_3(s_3+1) +2s_1s_3 -2(s_3+1) -2s_3 -s_4^2-\dots-s_{n-1}^2 \\
            & = 4s_1s_3 + 2s_3^2 -s_4^2-\dots-s_{n-1}^2 +2s_1-2s_3-2\\
            &\geq 2(s_2-s_3-1) \geq 0.
        \end{align*}
        As this is a contradiction, we must have that $s_1=s_2=s_3$.

        \item Suppose that $s_3\geq s_4+1$. Then, \eqref{fixed height ineq} implies 
        \begin{align*}
            0 &> 6(s_4+1)^2 - 4(s_4+1) -s_4^2-\dots-s_{n-1}^2\\
            & = 5s_4^2-s_5^2-\dots-s_{n-1}^2 + 2s_4+2 \geq 0
        \end{align*}
        As this is a contradiction, we must have that $s_1=s_2=s_3=s_4$.

        \item Suppose that $s_4\geq s_5+1$. Then, \eqref{fixed height ineq} implies
        \begin{align*}
            0 &> 5(s_5+1)^2 - 4(s_5+1) - s_5^2-\dots-s_{n-1}^2\\
             & = 4s_5^2 - s_6^2-\dots-s_{n-1}^2 +6s_5 + 1\\
             &\geq 6s_5+1 \geq 0.
        \end{align*}
        As this is a contradiction, we must have that $s_1=s_2=s_3=s_4=s_5$.

        \item Suppose that $s_5\geq s_6+1$. Then, \eqref{fixed height ineq} implies 
        \begin{align*}
            0 &> 4(s_6+1)^2-4(s_6+1) - s_6^2-\dots-s_9^2 \\
            & = 3s_6^2 - s_7^2-\dots-s_{n-1}^2  +4s_6-4 \geq 0.
        \end{align*}
        As this is a contradiction, we must have that $s_1=s_2=s_3=s_4=s_5=s_6$.
    \end{itemize}

This implies that we have no such sporadic points for $n\leq 6$. Suppose $n=7$. Then, any candidate tuple must be of one of the following form:
\begin{align*}
    (k,k,k,k,k,k,3k)  
\end{align*}
where $k\in\ZZ_{> 0}$. As $\bfs$ is assumed to be primitive, $k=1$ and the only possible tuple is $(1,1,1,1,1,1,3)$.

Suppose $n=8$. Then, any candidate tuple must be of one of the following forms:
\begin{align*}
    (k,k,k,k,k,k,s_7,3k) \\
    (k,k,k,k,k,k,k,3k) 
\end{align*}
where $k\in\ZZ_{> 0}$ and $s_7 \leq k-1$. As $\bfs$ is assumed to be primitive, the second possible form only contributes the tuple $(1,1,1,1,1,1,1,3)$. Suppose $\bfs$ is of the first form listed. We claim that $k=1$. By way of contradiction, suppose that $k\geq 2$. Then, $\bfs-(1,0,\dots,0,1)\in T_n$ as 
\begin{align*}
    (3k-1)^2-(k-1)^2-5k^2-(k+1) = 3k^2-5k+1 \geq 3 > 0
\end{align*}
Thus, the only tuple satisfying these restrictions is $(1,1,1,1,1,1,0,3)$.
 
Suppose $n=9$. Then, for $k\in\ZZ_{> 0}$, any candidate tuple must be of one of the following forms:
\begin{align}
    &(k,k,k,k,k,k,s_7,s_8,3k) \label{form 8.1}\\
    &(k,k,k,k,k,k,k,s_8,3k) \label{form 8.2}\\
    &(k,k,k,k,k,k,k,k,3k) \label{form 8.3}
\end{align}
where $s_7,s_8 \leq k-1$. 

Suppose $\bfs$ is of the form (\ref{form 8.1}). By way of contradiction, suppose that $k\geq 2$. We claim that $\bfs-(1,0,\dots,0,1)\in T_n$. This follows from that fact that 
\begin{align*}
    (3k-1)^2-(k-1)^2-5k^2-s_7^2-s_8^2 &\geq (3k-1)^2-(k-1)^2-5k^2-2(k-1)^2\\
    &= k^2-2 \\
    &\geq 2 >0
\end{align*}
Thus, $\bfs$ must not be sporadic and $k=1$. This gives us the tuple $(1,1,1,1,1,1,0,0,3)$. 

Suppose $\bfs$ is of the form (\ref{form 8.2}). By way of contradiction, suppose $k\geq 3$. Then, we claim that $\bfs-(1,0,\dots,0,1)\in T_n$. This follows from the fact that 
\begin{align*}
    (3k-1)^2-(k-1)^2 -6k^2 -s_8^2 &\geq (3k-1)^2-(k-1)^2 -6k^2 -(k-1)^2\\
    &=2k^2-2k-1 \\
    &\geq 2 > 0.
\end{align*}
Thus, we only need to consider $k=1,2$. If $k=1$, this gives us the tuple $(1,1,1,1,1,1,1,0,3)$. Suppose $k=2$. This gives us the following possible tuples: \[(2,2,2,2,2,2,2,0,6), \hspace{5mm}(2,2,2,2,2,2,2,1,6).\] This first tuple listed is not primitive so this only give us the tuple $(2,2,2,2,2,2,2,1,6)$. Lastly, if $\bfs$ is of the form (\ref{form 8.3}), then $\bfs$ is only primitive if $k=1$. This gives us our last tuple, $(1,1,1,1,1,1,1,1,3)$. 

Suppose $n=10$. Then,for $k\in\ZZ_{> 0}$, any candidate tuple must be of one of the following forms:
\begin{align}
    &(k,k,k,k,k,k,s_7,s_8, s_9, 3k) \label{10.1} \\
    &(k,k,k,k,k,k,k, s_8, s_9, 3k) \label{10.2}\\
    &(k,k,k,k,k,k,k,s_9,3k) \label{10.3}\\
    &(k,k,k,k,k,k,k,k,3k) \label{10.4}
\end{align}
where $s_7,s_8, s_9\leq k-1$. Any tuple of form (\ref{10.4}) is a Pythagorean tuple so we may exclude it. Suppose $\bfs$ is of the form (\ref{10.1}). By way of contradiction, suppose $k\geq 2$. Then, we claim that $\bfs-(1,0,\dots,0,1)\in T_n$. This follows from the fact that 
\begin{align*}
    (3k-1)^2-&(k-1)^2 -5k^2 -s_7^2 -s_8^2-s_9^2 \\
    &\geq (3k-1)^2-(k-1)^2 -5k^2 -3(k-1)^3\\
    &=2k-3 \geq 1 > 0.
\end{align*}
Thus, $k=1$ and the only tuple we have of this form is $(1,1,1,1,1,1,0,0,0,3)$. 

Suppose $\bfs$ is of the form (\ref{10.2}). By way of contradiction, suppose $k\geq 2$. If $s_9 \geq 1$, then $\bfs-(1,1,\dots,1,3)\in T_n$ as 
\begin{align*}
    (3k-3)^2 -7(k-1)^2-(s_8-1)^2-(s_9-1)^2 &\geq (3k-3)^2 -7(k-1)^2-2(k-2)^2\\
    &= 4k-6 \geq 2 > 0
\end{align*}
Thus $s_9=0$. Suppose $k\geq 3$. Then, $\bfs-(1,0,\dots,0,1)\in T_n$ as 
\begin{align*}
    (3k-3)^2 -(k-1)^2-6k^2-s_8^2 &\geq (3k-3)^2 -(k-1)^2-6k^2-(k-1)^2\\
    &= k^2-2k-1 \geq 2 > 0
\end{align*}
Thus, our options are $k=1,2$. If $k=1$, this recovers the tuple $(1,1,1,1,1,1,1,0,0,3)$. If $k=2$, this gives us potential tuples $(2,2,2,2,2,2,2,0,0,6)$ and $(2,2,2,2,2,2,2,1,0,6)$. The first is not primitive so we exclude it. 

Lastly, suppose $\bfs$ is of the form (\ref{10.3}). If $s_9\neq 0$, the $\bfs-(1,1,\dots,1,3)\in T_n$. This follows from the fact that 
\begin{align*}
    (3k-3)^2-8(k-1)^2-(s_9-1)^2 &\geq (3k-3)^2-8(k-1)^2-(k-2)^2\\
    &=14k-11 \geq 0
\end{align*}
Thus, $s_9=0$. The only primitive sporadic satisfying these constraints is $(1,1,1,1,1,1,1,1,0,3)$. This completes the proof.
\end{proof}


Then we show that besides the points listed in Lemma~\ref{lemma: stable height}, every other sporadic points will reduce to a strictly lower height after multiplication by $(A_n^+)^{-1}$.

\begin{lemma} \label{heightdrop<9}
Let $\bfs\in T_n$ be sporadic with nonnegative entries such that $s_1\geq s_2\geq\dots\geq s_{n-1}$, $s_1\ge 1$ and $3\leq n \leq 10$. For $s$ not listed in Lemma~\ref{lemma: stable height}, \[((A_n^+)^{-1} \bfs)_n < (\bfs)_n.\] 
\end{lemma}

\begin{proof}
    This is equivalent to showing that $-s_1-s_2-s_3+2s_n< s_n$, or rather $s_n<s_1+s_2+s_3.$ The case of $n=3$ reduces to the inequality $s_3<s_1+s_2$. By Lemma~\ref{ceiling}, we have that 
    \begin{align}
        s_n &= \bigg\lceil\sqrt{s_1^2+s_2^2+\dots+s_{n-1}^2}\bigg\rceil\notag\\
        &\le \bigg\lceil\sqrt{s_1^2+s_2^2+s_3^2 + 2s_1s_2+2s_1s_3+2s_2s_3}\bigg\rceil \label{dim9problem}\\ 
        &=  \bigg\lceil\sqrt{(s_1+s_2+s_3)^2}\bigg\rceil=s_1+s_2+s_3,\notag
    \end{align}
where the inequality follows from the order $s_1s_2\ge s_1s_3\ge s_2s_3\ge s_4^2\ge ...\ge s_{n-1}^2$ and $n\le 10$. As $s$ is not one of the tuples listed in Lemma~\ref{lemma: stable height}, the inequality \eqref{dim9problem} can be made strict. Therefore, $s_n<s_1+s_2+s_3$, which implies that $(A_n^+)^{-1}\bfs$ sits at a strictly lower height in the cone than $\bfs$.  
\end{proof}

\subsection{Proof of Theorem~\ref{thm:SOC}}
We now present a complete formulation of Theorem~\ref{thm:SOC} followed by its proof. 

\begin{theorem}\label{thm:SOC_full_version}
    For dimension $3\le n\le 10$, the conical semigroup $\mathrm{SOC}(n)\cap\ZZ^n$ is $(R,G)$-finitely generated by 
    \[G=\left\< A_n^+, Q_1, \dots, Q_{n-1}, P_{1,2}, P_{1,3},\dots P_{1,n-1} \right\>\] and a finite set $R$.
    More specifically,
    \begin{enumerate}
        \item If $3\leq n\leq 6$, then {\small$R =\left\{(1,0,\dots,0,1)^\transpose, (0,\dots,0,1)^\transpose \right\}$.}
        
        \item If $n=7$,  then {\small\[R =\left\{(1,0,0,0,0,0,1)^\transpose,(0,0,0,0,0,0,1)^\transpose, (1,1,1,1,1,1,3)^\transpose \right\}.\]}
        
        \item If $n=8$,  then 
        {\small\begin{align*}
             R =\Big\{
             (1,0,0,0,0,0,0,1)^\transpose,
             (0,0,0,0,0,0,0,1)^\transpose,
             (1,1,1,1,1,1,1,3)^\transpose,
             (1,1,1,1,1,1,0,3)^\transpose 
             \Big\}.  
        \end{align*}}
    
        \item If $n=9$,  then  
        {\small\begin{align*}
            R =\Big\{
            &(1,0,0,0,0,0,0,0,1)^\transpose,
            (0,0,0,0,0,0,0,0,1)^\transpose, 
            (1,1,1,1,1,1,1,1,3)^\transpose, \\
            &(1,1,1,1,1,1,1,0,3)^\transpose, 
            (1,1,1,1,1,1,0,0,3)^\transpose,
            (2,2,2,2,2,2,2,1,6)^\transpose  \Big\}.
        \end{align*}}

        \item If $n=10$, then
        {\small\begin{align*}
            R =\Big\{
            (1,0,0,0,0,0,0,0,0,1)^\transpose, (1,1,1,1,1,1,1,1,1,3)^\transpose, (0,0,0,0,0,0,0,0,0,1)^\transpose,\\ (1,1,1,1,1,1,1,1,0,3)^\transpose, (1,1,1,1,1,1,1,0,0,3)^\transpose,(1,1,1,1,1,1,0,0,0,3)^\transpose, \\(2,2,2,2,2,2,2,2,1,6)^\transpose, 
            (2,2,2,2,2,2,2,1,0,6)^\transpose
            \Big\}.  
        \end{align*}}      
    \end{enumerate}
\end{theorem}

\begin{proof}
    This follows directly from Lemma~\ref{heightdrop<9} and that fact that $(0,0,\dots,0,1)$ is the sporadic of minimal height in this cone. Let $\bfs\in T_n$. If $\bfs$ is not sporadic, we can represent it as 
    \begin{align} \label{soc: decomposition}
        \bfs= \lambda_1 \bfp_1 + \lambda_2 \bfp_2 + \dots + \lambda_k \bfp_k + \lambda \bfp
    \end{align}
   where $\lambda, \lambda_i\in \ZZ_{\geq 0}$, each $\bfp_i$ is a primitive Pythagorean tuple and $\bfp$ is sporadic. By Lemma~\ref{cass (R,G)}, each $\bfp_i$ can be decomposed as $\bfp_i= G_i(1,0,\dots,0,1)^\transpose$ when $3\leq n < 10$ or $\bfp_i= G_i(1,0,\dots,0,1)^\transpose + \tilde{G_i}(1,1,\dots,1,3)^\transpose$ when $n=10$ where each $G_i,\tilde{G_i}\in G$. It remains to consider the sporadic $\bfp$. Given any primitive sporadic tuple $\bfs$, we can recover an element of $R$ as follows:
    \begin{enumerate}
        \item Multiply $\bfp$ by the appropriate permutation matrices $P_{i,j}$ and sign changing matrices $Q_j$ so that $\bfp$ has nonnegative entries and $p_1\geq\dots\geq p_{n-1}$. Call this resulting vector $\bfp'$.
        \item Multiply $\bfp'$ by $(A_n^+)^{-1}$ and repeat step 1 as necessary. By Lemma~\ref{heightdrop<9}, the height of the resulting vector will be strictly lower then that of the vector we started with or the resulting vector will belong to $R$.
        \item Repeat step 2 until the resulting vector $\bfr$ belongs to $R$. By Lemma~\ref{lemma: stable height}, the only possibilities for the resulting vector belong to $R$.
    \end{enumerate}
    This process gives the equality $\bfr=G_1\dots G_k \bfp$. If we let $G'=G_1\dots G_k$, then we have $(G')^{-1}\bfr=\bfp$.  Therefore, for $3\leq n \leq 10$, the conical semigroup $\text{SOC}(n)$ is $(R,G)$-finitely generated by the claimed $R$ and $G$.
\end{proof}

When $n=9$, the primitive sporadic point $(2,2,2,2,2,2,2,1,6)$ can be written as the sum of two sporadic points with smaller heights:
{\small\[(2,2,2,2,2,2,2,1,6) = (1,1,1,1,1,1,1,0,3)+             (1,1,1,1,1,1,0,0,3).\]} We can similarly decompose $(2,2,2,2,2,2,2,2,1,6)$ and $(2,2,2,2,2,2,2,1,0,6)$ for $n=10$. In this sense, these sporadic points fail to be minimal. Thus, if we remove them from the set of roots $R$, our semigroup $S$ remains $(R,G)$-finitely generated. However, when we remove these point from our root sets, our decomposition in equality (\ref{soc: decomposition}) requires modification and we must allow for multiple sporadic points in the expression.

\begin{remark}
    Lastly, it is worth noting that inequality \eqref{dim9problem} would fail in dimensions larger than 10. Thus, this line of argumentation would fail to produce results for $n>10$. 
\end{remark}


We can use Theorem~\ref{thm:SOC_full_version} to recover a partial converse of Lemma~\ref{sporadic=-1}. 

\begin{corollary}
    Let $3\leq n < 7$ and fix $\bfs\in T_n$. If $\bfs$ is a primitive sporadic, then $\<\bfs,\bfs\>=-1$.
\end{corollary}

\begin{proof}
    Let $\bfs\in T_n$ be sporadic. Using theorem~\ref{thm:SOC_full_version}, we can express $\bfs$ as \[\bfs= G'(0,\dots,0,1)^\transpose\] for $G'\in G$. As $\<G' \bfs, G' \bfs\>=\<\bfs,\bfs\>$ for all $G'\in G$, we see that \[\<\bfs,\bfs\>=\<G'(0,\dots,0,1)^\transpose, G'(0,\dots,0,1)^\transpose\>= \<(0,\dots,0,1)^\transpose, (0,\dots,0,1)^\transpose\>=-1.\]
\end{proof}    

In particular, this converse fails to hold in dimensions $\ge 7$ as we have $\bfr\in R$ such that $\<\bfr,\bfr\>\neq -1$.


\section{Not All Conical Semigroups are $(R,G)$-Finitely Generated} \label{notallconesrgfg}

The next example shows that not every conical semigroup is $(R,G)$-finitely generated.

\begin{example}\label{ex:NonFinGen}
    Let $C:=\{(x,y)\in\RR^2:0\le y\le \alpha x\}$ be a cone generated by the vectors $v_1:=(1,0)$ and $v_2:=(1,\alpha)$, where $\alpha>0$ is an irrational number.
    We claim that the conical semigroup $S_C:=C\cap\ZZ^2$ is not $(R,G)$-finitely generated for any finitely generated group $G$ acting linearly on $C$ and any finite set $R\subset S_C$.
    {\color{blue}
    It suffices to show that any $g\in G$ that acts linearly on $C$ and restricts to $S_C$ must be the identity. Specifically, if $g$ is such that $g(c_1u_1+c_2u_2)=c_1g(u_1)+c_2g(u_2)$ for any $u_1,u_2\in\RR_{\ge0}$ and $u_1,u_2\in C$, and $g(S_C)\subseteq S_C$, then $g=I$.}


    
    To see this, note that $C$ is simplicial as $v_1$ and $v_2$ are linearly independent, so any vector $v\in C$ can be written uniquely as a linear combination of $v_1$ and $v_2$.
    Thus any {\color{blue}linear action by $g$} on $C$ can be represented by a matrix
    \[
        U_g=\begin{bmatrix}
            u_{11} & u_{12} \\
            u_{21} & u_{22}
        \end{bmatrix}\in\RR^{2\times2}.
    \]
    The determinant $\det{U_g}=u_{11}u_{22}-u_{21}u_{12}$, which is nonzero due to the definition of the group action $g$, and $U_g^{-1}=\dfrac{1}{\det{U_g}}\begin{bmatrix} u_{22} & -u_{12} \\ -u_{21} & u_{11}\end{bmatrix}$.
    Since $v_1=(1,0)\in S_C$, we must have 
    \[
        U_g v_1=\begin{bmatrix}
            u_{11} \\ u_{21}
        \end{bmatrix}\in S_C,
    \]
    which implies that $u_{11},u_{21}\in\ZZ$ and 
    \begin{equation}\label{ex:ineq_1}
        0\le u_{21}\le\alpha u_{11}.
    \end{equation}
    Let $w:=(\lceil1/\alpha\rceil,1)\in S_C$, we must have
    \[
        U_g {\color{blue}w}=\begin{bmatrix}
            \lceil1/\alpha\rceil u_{11}+u_{12} \\ \lceil1/\alpha\rceil u_{21}+ u_{22}
        \end{bmatrix}\in S_C,
    \]
    which implies that $u_{12},u_{22}\in\ZZ$.
    At this point, we see that $U_g\in\ZZ^{2\times2}$.
    Now take a sequence of vectors $w_k:=(x_k,y_k)\in S_C$ such that $\lim_{k\to\infty}\dfrac{y_k}{x_k}=\alpha$, the existence of which is guaranteed by the Dirichlet's approximation theorem.
    Since $U_g w_k\in S_C$, we must have $0\le u_{21}x_k+u_{22}y_k\le\alpha(u_{11}x_k+u_{12}y_k)$, which in turn gives \begin{equation}\label{ex:ineq_2}
       0\le u_{21}+\alpha u_{22}\le\alpha(u_{11}+\alpha u_{12}) 
    \end{equation} 
    by taking $k\to\infty$.
    Moreover, since $$U_g^{-1}v_1=\frac{1}{\det{U_g}}\begin{bmatrix} u_{22} \\ -u_{21}\end{bmatrix}\in S_C,$$ 
    we must have 
    \begin{equation}\label{ex:ineq_3}
    0\le(\det{U_g})^{-1}(-u_{21})\le\alpha(\det{U_g})^{-1}u_{22},
    \end{equation}
    which implies that
    $\dfrac{1}{\det{U_g}}(\alpha u_{22}+u_{21})\ge0$.
    By~\eqref{ex:ineq_2}, $u_{21}+\alpha u_{22}\ge0$, and the inequality is strict because of the irrationality of $\alpha$.
    Consequently, \eqref{ex:ineq_3} shows that $\det{U_g}>0$.
    Thus~\eqref{ex:ineq_3} further implies that $u_{21}\le0$, so $u_{21}=0$ by~\eqref{ex:ineq_1}.
    Now consider 
    $$U_g^{-1}w_k=\dfrac{1}{\det{U_g}}\begin{bmatrix} u_{22}x_k-u_{12}y_k \\ u_{11}y_k \end{bmatrix}\in S_C,$$
    so again by taking $k\to\infty$, we have $0\le\alpha u_{11}\le\alpha(u_{22}-\alpha u_{12})$.
    Together with~\eqref{ex:ineq_2}, which becomes $u_{22}\le u_{11}+\alpha u_{12}$, we see that $u_{22}=u_{11}+\alpha u_{12}$.
    Again using the irrationality of $\alpha$, we must have $u_{12}=0$, and consequently $u_{11}=u_{22}$.
    Finally, note that $U_g^{-1}v_1=\begin{bmatrix} 1/u_{11} \\ 0\end{bmatrix}\in S_C$, so $1/u_{11}\in\ZZ$ and thus $u_{11}=\pm1$.
    From~\eqref{ex:ineq_1}, we conclude that $u_{11}=u_{22}=1$, and $U_g=\begin{bmatrix} 1 & 0 \\ 0 & 1 \end{bmatrix}$ is the identity matrix.
    Since the generator $v_2$ is not a rational vector, we know that $S_C$ itself is not finitely generated, so it is not $(R,G)$-finitely generated either.
\end{example}

\section{Applications: Total Dual Integrality, Chvátal-Gomory Closures, and Integer Rank} \label{applications}
{\color{blue} To prove Theorems~\ref{thm:TDI} and~\ref{thm:ICR}, we use the notation \(\bbR^{\oplus I}\) (or simply \(\bbR^I\)) to denote an \(\bbR\)-vector space for a possibly infinite index set $I$, where each vector \((a_i)_{i\in I}\in\bbR^I\) has all but finitely many zero coordinates \(a_i=0\). }

\begin{proof}[Proof for Theorem~\ref{thm:TDI}]
    The containment \(\subseteq\) follows from definition.
    To see the other containment, take any halfspace \(H:=\{\bfx\in\bbR^m:{\color{blue}\bfu}^\transpose \bfx\le v\}\) for some \((\bfu,v)\in\bbZ^m\times\bbR\) such that \(C\subseteq H\).
    Then by the full dimensionality of \(C\), we have
    \begin{equation*}
        \begin{aligned}
        v\ge&\sup_{x}\left\{\bfu^\transpose \bfx:\bfc-\calA(\bfx)\in C\right\}
         =\inf_{y}\left\{y(\bfc):\calA^*(y)=\bfu,\ y\in C^*\right\}.
         \end{aligned}
    \end{equation*}
    Using the TDI assumption, the infimum is attained by some \(y^*\in C^*\cap\bbZ^N\).
    As {\color{blue}$S_{C^*}$} is $(R,G)$-finitely generated, $r_1,\dots,r_k\in R$, $g_1,\dots,g_k\in G$, and \(\lambda_1,\dots,\lambda_k\in\bbZ_{\ge0}\) such that
    \[
        y^*=\sum_{j=1}^k \lambda_j(g_j\cdot r_j),
    \]
    for some \(k\ge 1\).
    Consequently, we have
    \begin{equation*}
        \lfloor v\rfloor\ge\lfloor y^*(\bfc)\rfloor=\biggl\lfloor\sum_{j=1}^k \lambda_j(g_j\cdot r_j)^\transpose \bfc\biggr\rfloor\ge\sum_{j=1}^k \lambda_j\lfloor(g_j\cdot r_j)^\transpose \bfc\rfloor.
    \end{equation*}
    {\color{blue} Also as a consequence of the \((R,G)\)-finite generation of $S_{C^*}$, determining $\inf_{y}\left\{y(\bfc):\calA^*(y)=\bfu,\ y\in C^*\right\}$ is equivalent to solving the following (semi-infinite) linear optimization problem:
    \begin{align*}
        \inf_{\lambda}\quad & \sum_{r\in R,g\in G}\lfloor(g\cdot r)^\transpose \bfc\rfloor \lambda_{r,g}\\
        \mathrm{s.t.}\quad & \calA^*\biggl(\sum_{r\in R,g\in G}\lambda_{g,r}(g\cdot r)\biggr) =\bfu,\\
        & \lambda\in \bbR^{\oplus{\{(g,r):~r\in R, g\in G\}}}_{\ge0}.
    \end{align*}
    Note that \((\lambda_1,\dots,\lambda_k)\) is a feasible solution to this problem. By weak duality of the semi-infinite optimization problem, we have
    \begin{equation*}
        \sum_{j=1}^k \lambda_j\lfloor(g_j\cdot r_j)^\transpose \bfc\rfloor\ge
        \sup_x\left\{\bfu^\transpose \bfx:\calA^*(g\cdot r)^\transpose \bfx\le\lfloor(g\cdot r)^\transpose \bfc\rfloor,\ \forall\,r\in R,g\in G\right\}.
    \end{equation*}
    }
    Therefore, the inequality \(\bfu^\transpose \bfx\le \lfloor v\rfloor\) is implied by the inequalities \((g\cdot r)^\transpose \calA(\bfx)=\calA^*(g\cdot r)^\transpose \bfx\le\lfloor(g\cdot r)^\transpose \bfc\rfloor\) for \(r\in R,g\in G\).
    Since the halfspace \(H\) is arbitrary, we conclude that 
    \[
        \CGclos(Z)\supseteq\left\{\bfx\in\bbR^m:(g\cdot r)^\transpose\calA(\bfx)\le\lfloor(g\cdot r)^\transpose \bfc\rfloor,\quad\forall\,r\in R,g\in G\right\}. \qedhere
    \]
\end{proof}

\begin{proof}[Proof for Theorem~\ref{thm:ICR}]
    Suppose \(B=\{b_i\}_{i\in I}\subset S_C\) for some possibly infinite index set \(I\) is an integer generating set for \(S_C\).
    
    Consider a semi-infinite linear optimization problem
    \begin{equation}\label{eq:ICR_LP}
        \begin{aligned}
            \max\quad&\sum_{i\in I}\lambda_i\\
            \mathrm{s.t.}\quad& \sum_{i\in I}\lambda_ib_i=s,\\
                              & (\lambda_i)_{i\in I}\in\bbR^{\oplus I}_{\ge0}.
        \end{aligned}
    \end{equation}
    Since \(C\) is pointed, \(B\) satisfies the ``opposite sign condition,'' meaning that whenever \(\sum_{i\in I}\mu_ib_i=0\) for some nonzero \((\mu_i)_{i\in I}\in\bbR^{\oplus I}\), we have \(\mu_i<0<\mu_j\) for some \(i,j\in I\).
    Thus by~\cite[Theorem 2]{charnes1963duality}, we know that~\eqref{eq:ICR_LP} has an extreme point solution, denoted as \((\lambda^*_i)_{i\in I}\) with \(J:=\{i\in I:\lambda^*_i>0\}\).
    By~\cite[Theorem 1]{charnes1963duality}, the vectors \(\{\lambda_i\}_{i\in J}\subset S_C\) associated with the extreme point solution must be linearly independent.
    Thus \(\vert J\vert\le N\).

    For each \(i\in J\), let \(z_i:=\lfloor\lambda_i^*\rfloor\) and \(y_i:=\lambda_i^*-z_i\).
    We claim that \(\sum_{i\in J}y_i<N-1\).
    Given this claim, the theorem is proved as follows.
    The vector \(s':=s-\sum_{i\in J}z_ib_i\in C\cap\bbZ^N=S_C\) can be written as integer combination of \(B\) by definition, that is, there exists an index set \(J'\subset I\) with \(b'_i\in B\), \(\lambda'_i\in\bbZ_{\ge1}\) for each \(i\in J'\) such that \(s'=\sum_{i\in J'}\lambda'_ib'_i\).
    This implies that 
    \begin{equation*}
        s=s'+\sum_{i\in J}z_ib_i=\sum_{i\in J}z_ib_i+\sum_{j\in J'}\lambda'_ib'_i.
    \end{equation*}
    We see that \(\sum_{i\in J}z_i+\sum_{j\in J'}\lambda'_i\le\sum_{i\in J}\lambda^*_i\) by the optimality of \((\lambda^*_i)_{i\in I}\).
    Consequently, \(\sum_{j\in J'}\lambda'_j\le \sum_{i\in J}y_i<N-1\), and thus \(s\) can be written as an integer sum of at most \(\vert J\vert+N-2\le 2N-2\) generators from \(B\).

    It remains to prove the claim, \(\sum_{i\in J}y_i<N-1\). 
    Note that if \(\vert J\vert\le N-1\) this is trivially true because \(y_i<1\) by definition.
    So we may assume that \(\vert J\vert=N\) and denote \(J=\{1,\dots,N\}\) without loss of generality.
    Moreover, if the convex hull \(V:=\conv\{b_1,\dots,b_N\}\) has a nonempty intersection with \(B\), say \(b_{N+1}\in V\cap B\) with \(b_{N+1}=\sum_{i=1}^{N}\gamma_ib_i\) for some \(0<\gamma_1,\dots,\gamma_N<1\), \(\sum_{i=1}^{N}\gamma_i=1\), then we can write \(s\) as
    \begin{equation*}
        s=\epsilon b_{N+1}+\sum_{i=1}^{N}(\lambda_i^*-\epsilon\gamma_i)b_i,
    \end{equation*}
    where \(\epsilon:=\frac{\lambda_\iota^*}{\gamma_\iota}\) for some \(\iota\in\mathrm{argmin}\{\frac{\lambda_i^*}{\gamma_i}:i=1,\dots,N\}\).
    This shows that \((\mu^*_i)_{i\in I}\) with \(\mu_i^*:=\lambda^*_i-\epsilon\gamma_i\) for each \(i\in J\setminus\{\iota\}\), \(\mu_{N+1}^*:=\epsilon\), and \(\mu_i^*=0\) for any \(i\notin J_1:=J\cup\{N+1\}\setminus\{\iota\}\), is also an optimal solution to~\eqref{eq:ICR_LP}.  
    Thus by replacing \(J\) with \(J_1\), we can assume that the intersection \(V\cap B=\varnothing\).
    Under this assumption, let \(s'':=\sum_{i=1}^{N}(1-y_i)b_i=\sum_{i=1}^{N}b_i-s'\in C\cap\bbZ^N\).
    Again by the definition of \(B\), we can write \(s''=\sum_{j\in J''}\lambda_j''b_j''\), for some finite subset \(J''\subset I\), \(b_j''\in B\) and \(\lambda''_j\in\bbZ_{\ge1}\) for each \(j\in J''\).
    Note that
    \begin{equation*}
        s=\delta s''+\sum_{i\in J}(\lambda_i^*-\delta(1-y_i))b_i
        =\delta\sum_{i\in J''}\lambda_i''b_i+\sum_{i\in J}(\lambda_i^*-\delta(1-y_i))b_i,
    \end{equation*}
    for some sufficiently small \(\delta>0\),
    so by the optimality of \((\lambda_i^*)_{i\in I}\), we must have \(1\le\sum_{i\in J''}\lambda_i''\le\sum_{i\in J}(1-y_i)\).
    If \(\sum_{i=1}^{N}y_i\ge N-1\), then this implies that \(\sum_{i\in J}(1-y_i)=1\), which is a contradiction with our assumption \(V\cap B=\varnothing\).
    Thus we must have \(\sum_{i=1}^{N}y_i< N-1\).
\end{proof}

\begin{remark}
    If we apply the theorem to the case \(S_C=\calS_+^n(\bbZ)\), then \(N=\dim{\calS^n(\bbR)}=\binom{n+1}{2}\).
    The bound on the ICR in this case is \(2N-2=n^2+n-2\), which grows \emph{quadratically} with \(n\) as opposed to the linear growth in the case of the usual Carathéodory rank of positive semidefinite matrices.
    If we apply the theorem to the case \(T_n=\mathrm{SOC}(n)\cap\bbZ^n\), then \(N=\dim{\mathrm{SOC}(n)}=n\).
    The ICR in this case is \(2N-2=2n-2\).
\end{remark}

\bibliography{isdp}

\end{document}